\documentclass[11pt]{amsart}
\usepackage{amssymb,amsmath,amsthm,mathrsfs,epsf,float,cancel,enumerate}
\usepackage{graphics}
\usepackage{graphicx}
\usepackage{epstopdf}
\usepackage{verbatim}

\newtheorem{theorem}{Theorem}

\newtheorem{fact*}{Fact*}

\newcommand{\ol}{\overline}

\newcommand{\cT}{\mathcal{T}}
\newcommand{\oP}{\operatorname{Pr}}
\newcommand{\op}{\omega^{\prime}}
\newcommand{\cG}{\mathcal{G}}
\newcommand{\cE}{\mathcal{E}}

\allowdisplaybreaks

\begin{document}

\parindent = 0cm
\parskip = .3cm

\title[Property O]{The minimum Number of Edges\\in Uniform Hypergraphs with Property O}

\author[D. Duffus, B. Kay, V. R\"odl]{Dwight Duffus, Bill Kay and  Vojt\v{e}ch R\"{o}dl$^1$}
\address {Mathematics \& Computer Science\\
   Emory University, Atlanta, GA  30322 USA}
   \email[Dwight Duffus]{dwight@mathcs.emory.edu}
   \email[Bill Kay]{bill.w.kay@gmail.com}
   \email[ Vojt\v{e}ch R\"{o}dl]{rodl@mathcs.emory.edu}

\keywords{oriented $k$-uniform hypergraph}

\subjclass[2010]{05C65, 05C35, 05D40}

\footnotetext[1]{V. R\"odl was supported by NSF grant DMS 1301698}

\begin{abstract}
An oriented $k$-uniform hypergraph (a family of ordered $k$-sets) has the ordering property (or Property O) if for every 
linear order of the vertex set, there is some edge oriented consistently with the linear order.  We find bounds on the minimum 
number of edges in a hypergraph with Property O.
\end{abstract}

\date{10 August 2016}
\maketitle
\thispagestyle{empty}

\section{Introduction}\label{S:intro}

This note is motivated by two types of problems concerning hypergraphs.  The first is well-known and regards
2-colorable hypergraphs, also said to possess Property B.  Several papers have presented bounds on $m(k)$,
the minimum number of edges in a $k$-uniform hypergraph that does not have Property B (see
\cite{B}, \cite{CK}, \cite{RS} and \cite{S}).  The second comes from Ramsey theory, where appropriate properties
of graphs containing a given graph with a fixed order can be used to prove negative partition relations for unordered graphs (see \cite {NR} and \cite{NR2} for early papers on this topic). 

We would like to determine the minimum number of edges in a oriented uniform hypergraph needed to ensure that for
every ordering of the vertex set, some edge is ordered in the same way.  Here are the required definitions 
followed by our results and a conjecture.

Fix a positive integer $k \ge 2$ and a finite set $V$.  An {\it ordered $k$-set} $\overline{E}$ is a $k$-tuple $(x_1, x_2, \ldots, x_k )$ 
of distinct elements of $V$; we use $E$ to denote the unordered set $\{x_1, x_2, \ldots, x_k\}$.  Given a family of ordered $k$-sets 
$\mathcal{E}\subseteq V^k$ with no two $k$-tuples on the same $k$-element set, call $\mathcal{H} = (V, \mathcal{E})$  an 
{\em oriented $k$-uniform hypergraph}, or, more briefly, an {\it oriented $k$-graph}.  In the case that $\mathcal{E}$ contains 
an ordered $k$-set for each $k$-element subset of $V$, call $\mathcal{H}$ 
a {\it $k$-tournament}.   So, a $k$-tournament is obtained from the complete $k$-uniform hypergraph $K_n^{(k)}$ by giving each 
$k$-set an orientation.   For $E \subseteq V$ and a linear order $<$ on $V$, an ordered $k$-set $\ol{E} = (x_1, x_2, \ldots x_k)$
is {\it consistent} with $<$ if $x_1 < x_2< \ldots < x_k$. 

Here is the property that interests us.  

  {\bf Definition.}
     Given an oriented  $k$-graph $\mathcal{H} = (V, \mathcal{E})$ we say that that $\mathcal{H}$ 
     has the {\it ordering property}, or {\em Property O}, if for every (linear) order $<$ of $V$ there exists $\ol E \in \mathcal{E}$ that is consistent with $<$.   For an 
     integer $k \ge 2$, let \\[-.2cm]
     
 \centerline{$f(k)$ be the minimum number of edges in an oriented $k$-graph with Property O.}

Here is what we know about bounds for $f(k)$.\\

  \begin{theorem}\label{T:ub}
     The function $f(k)$ satisfies $k! \le f(k) \le (k^{2} \ln k) k!$ where the lower bound holds for all $k$ and the upper bound for $k \ge k_0$.
  \end{theorem}

The upper bound for $f(k)$ is proven in Section \ref{S:f(k)}.    The lower bound $k! \le f(k)$ follows from a standard 
argument.  Given any oriented $k$-graph  $\mathcal{H} = ([n], \mathcal{E})$, clearly each $\ol{E} \in \mathcal{E}$ is consistent
with 
$$(n - k)! \dbinom{n}{k} = \dfrac{n!}{k!}$$
orders on $[n]$.  Consequently, if $\mathcal{H}$ has Property O then 
$$|\mathcal{E}| \cdot\dfrac{n!}{k!} \ \ge \ n!, \ \text{so} \ |\mathcal{E}|  \ge k!.$$ 

We would like to decide if $f(k)$ is bounded away from $k!$, in analogy with Property~B.

 {\bf Problem 1.}  Determine whether  $\dfrac{f(k)}{k!} \to \infty$ as $k \to \infty$.

 We are unable to improve the simple lower bound for $f(k)$ at this point, however we can show that for any function $\alpha(k) \to 0$
 as $k \to \infty$, almost all $k$-tournaments with $$\binom{n}{k} = (1 - \alpha(k))(k^{1/2} k!)$$ edges fail to have Property O.     
 Let $\cT_{n,k}$ denote the set of all $k$-tournaments on $[n]$.  
 
   \begin{theorem}\label{T:random} 
     Let $0 < \alpha < 1$ and let $c = 2\pi/e^{1 + e^2/2}$.  If $n = (c \alpha (1 + o(1))^{1/k} \left( \frac{k}{e} \right)^2 k^{3/2k}$
     then for $k$ sufficiently large at least  $(1-\alpha)|\cT_{n,k}|$ members of $\cT_{n,k}$ do not have Property O. 
  \end{theorem}
  
In the next section, we prove the upper bound of $f(k)$ given in Theorem \ref{T:ub}.  In Section \ref{S:random} we prove Theorem \ref{T:random}.  In Section \ref{S:prob} we provide a construction of $k$-graphs with Property O, investigate the situation for small values of  $n$ and $k$, and pose a few problems.  

We close this section with an observation used in both Sections  \ref{S:f(k)} and \ref{S:random}.\\

 {\bf Fact.} Let  
      \begin{equation}
     n = {\left( \frac{k}{e} \right)}^2 (1 + o(1)) \label{e:nbound} .\\
  \end{equation}  
 Then 
  \begin{equation}   
     \binom{n}{k} = \left(e^{-e^2/2}\right) \frac{n^k}{k!}(1 + o(1)) \label{e:nkbound} .\\
   \end{equation}

 Indeed, (\ref{e:nbound}) implies that\\[-.2cm]
  \begin{align*}
    & \frac{(n)_k}{n^k} \ = \ \left(1 - \frac{1}{n}\right)\left(1 - \frac{2}{n}\right) \cdots \left(1 - \frac{k - 1}{n}\right) \ = \
   \exp \left( \sum_{j = 1}^{k - 1} \ln(1 - j/n) \right) \ = \\
    &= \exp \left((1 + o(1))\sum_{j = 1}^{k - 1} -j/n \right) \ = \ \exp \left( (1 + o(1)) \left(-\binom{k}{2}/n \right) \right) \ = \\
    &= \exp\left((1 + o(1))\left(- e^2/2 \right)\right) \ = \ (1 + o(1)) e^{-e^2/2} ,\\[-.2cm]
  \end{align*}
 so  (\ref{e:nkbound}) holds.

\section{Proof of Theorem \ref{T:ub}}\label{S:f(k)}

We verify the upper bound in Theorem \ref{T:ub} by showing that for $k$ large enough, there exists a $k$-tournament with 
$(k^2 \ln k) k!$ edges which has Property O.  Indeed, we show that for an appropriate choice of $n$, a randomly selected 
member of $\cT_{n,k}$ has Property O with positive probability. 

Let $\mathcal{H} = ([n], \mathcal{E}) \in \mathcal{T}_{n, k}$. For a fixed  order $<$ on $[n]$ and a fixed  $\overline{E}\in \mathcal{E}$, the 
probability that $\ol E$ is not consistent with $<$ is $1-\frac{1}{k!}$.  Since the edges of $\mathcal{H}$ are oriented independently, 
the probability that no edge of $\mathcal{H}$ is consistent with $<$ is $(1-\frac{1}{k!})^{\binom{n}{k}}$. Taking the union bound over 
all orders on $V$, we see that the probability that there exists an order $<$ on $V$ so that no edge of $\mathcal{H}$ is consistent 
with $<$ is at most  $n!(1-\frac{1}{k!})^{\binom{n}{k}}$.
 
The upper bound follows once we verify (\ref{e:first}) and (\ref{e:second}), below, for $k$ sufficiently large.
 
Let $n = \left(\dfrac{k}{e}\right)^2\left( \pi \cdot \exp (e^2/2)\cdot k^3 \ln k \right)^{1/k}$. Then
   \begin{equation}\label{e:first}
     \binom{n}{k} \frac{1}{k!} \  \le \  k^2 \ln k, \ \text{and}\\ 
   \end{equation}
   \begin{equation}\label{e:second}  
     n!\left( 1 - \frac{1}{k!} \right)^{\scalebox{1.2}{$\binom{n}{k}$}} \ < \ 1.
   \end{equation}

To prove (\ref{e:first}), we apply the Fact and the Stirling approximation $k! = (k/e)^k \sqrt{2 \pi k}(1 + o(1))$:
  \begin{equation}\label{e:ub}
       \binom{n}{k} \frac{1}{k!} \  = \  e^{-e^2/2} \left( \frac{\left( {k}/{e} \right)^{2k} \pi \cdot e^{e^2/2} \cdot k^3 \ln k}{(k!)^2}\right) (1 + o(1)) \
       = \ \frac{1}{2} k^2 \ln k (1 + o(1)).
  \end{equation}
Hence (\ref{e:first}) holds for $k$ sufficiently large.

 Turning to inequality (\ref{e:second}),  we use the choice of $n$ and (\ref{e:ub}) to infer that
  \begin{equation}\label{e:nln}
    \begin{aligned}
     n \ln n \ &= \ 2 \left(\dfrac{k}{e} \right)^2 \ln k (1 + o(1)) \ < \  \binom{n}{k} \frac{1}{k!},
    \end{aligned}
  \end{equation}
for $k$ sufficiently large.   We have
   \begin{align*}
    n! \left(1 - \frac{1}{k!} \right)^{\scalebox{1.2}{$\binom{n}{k}$}} \ &\le  \ n^n \left(1 - \frac{1}{k!}\right)^{\scalebox{1.2}{$\binom{n}{k}$}}
    \\
                                                      &\le \  \exp(n \ln n) \cdot \exp\left(-\binom{n}{k}\frac{1}{k!}\right) \\
                                                      &=    \ \exp\left(n \ln n -\binom{n}{k}\frac{1}{k!}\right)  \\
                                                      & < \ 1 .                
  \end{align*}
  where the last inequality follows from (\ref{e:nln}).  This proves (\ref{e:second}) and completes the proof of the upper bound.
  
\section{Randomly Oriented Tournaments}\label{S:random}

In this section we prove Theorem \ref{T:random}.  For brevity, we set $\omega =  (\alpha/3e) k^{1/2}(1 + o(1))$.

\begin{proof}[Proof of Theorem \ref{T:random}]
Let $0 < \alpha < 1$.  We refine the Fact, that is, the estimate of the number of edges of the $k$-tournaments in $\cT_{n,k}$.
Since
   $$n = (c \alpha (1 + o(1))^{1/k} \left( \frac{k}{e} \right)^2 k^{3/2k} \ \text{where} \ c = \frac{2 \pi}{3 e} e^{e^2/2} ,$$
by Stirling's formula we have
  \begin{equation}\label{e:n^k1}
    \begin{aligned}      
      n^k \ &= \  c \alpha k^{3/2} \left(\frac{k}{e}\right)^{2k}(1 + o(1))\\
               &= \  \frac{e^{e^2/2}}{3e} \alpha (k!)^2 k^{1/2} (1 + o(1)) .
     \end{aligned}
   \end{equation}
On the other hand, by the Fact,
    \begin{equation}\label{e:n^k2}
       n^k \ = \ e^{e^2/2} (n)_k (1 + o(1)).
    \end{equation}
Comparing the right hand sides of (\ref{e:n^k1}) and (\ref{e:n^k2}), we obtain
\begin{equation}\label{e:omega}
  \binom{n}{k} \ = \ \frac{\alpha}{3e} k^{1/2} k! (1 + o(1)) \ = \ \omega k!
\end{equation}
with $\omega$ as defined above. 

 We will show that if $T$ is sampled from $\cT_{n,k}$, the set of all $k$-tournaments on $[n]$, according to the uniform 
 distribution, the probability that $T$ has Property O is at most $\alpha$. It will follow that at least $(1-\alpha)|\cT_{n,k}|$ members of $\cT_{n,k}$ fail to have Property O. 
 
The random sampling of $T = ([n], \mathcal{E})$ from $\cT_{n,k}$ is done in two phases.  
In the first phase we will select $k$-tuples which are consistent with the natural order $<$ on $[n]$ and in the
second phase we will assign to the remaining $k$-tuples one of the $k! - 1$ remaining orientations.

   {\bf Phase 1}: Reveal the set $C(T)$ of the members of $\mathcal{E}$ that are oriented consistently with  $<$.  

For any $\ol{K}\in \mathcal{E}$, $\oP(\ol{K} \in C(T)) = 1/k!$ and thus, by (\ref{e:omega}),
     $$E(|C(T)|) \ = \ \binom{n}{k} \frac{1}{k!} \ = \ \omega.$$     
Let $A_\omega$ be the event that $|C(T)| \leq \frac{2}{\alpha}\omega$. By Markov's inequality we have
     $$\oP\left(|C(T)| >\frac{2 }{\alpha}\omega\right) \ < \ \frac{\alpha}{2} \ \ \text{and so} \ \ \oP(A_\omega) > 
     1 - \frac{\alpha}{2}.$$

Assume that $A_\omega$ occurs.   For each $\ol{K} \in \mathcal{E}$, as before, let $\min K$ be the $<$-least element of $K$.  Define
     $$M = \{  \min K \ | \  \ol{K}\in C(T)\}$$
and note that    
     $$|M| \leq |C(T)|  \leq \frac{2}{\alpha}\omega\ < k-1.$$
Thus, for each $\ol{K} \in C(T)$, $K \setminus M \ne \emptyset$.  Let $W \subseteq [n]$ be obtained by selecting one 
element from each $K \setminus M$. We now define $<^\prime$ to be the natural order $<$ on each of $W$ and 
$[n] \setminus W$, and let $u <^\prime v$ for $u \in W$, $v \in [n]\setminus W$.

We claim that no $\ol{K} \in C(T)$ is consistent with $<^\prime$. To see this, let $v \in K \cap W$. On the one hand, 
$\min K \not \in W$ by the way that we selected   $W$. On the other hand, $v <^\prime \min K$ by the definition of $<^\prime$. 
However, $\ol{K} \in C(T)$ means precisely that $\ol{K}$ is consistent with $<$, and so $v <^\prime \min K$ is a 
contradiction. 

  {\bf Phase 2:}  Reveal the orientation on each  $\ol{K} \not \in C(T)$.

For each $\ol{K} \not \in C(T)$ there are $k! - 1$ possible orientations of $\ol{K}$ in $T$ --  any one 
except that given by the natural order $<$.   Since $T$ is chosen according to the uniform distribution, each
orientation is equally likely and at most one of these  is consistent with $<^\prime$.   Thus, 
  $$\oP(\ol{K} \ \text{is consistent with} \  <^\prime) \le 1/(k! - 1).$$
Also, if $K \cap W = \emptyset$, then $<$ and 
$<^{\prime}$ coincide on $K$, so $\ol{K}$ cannot be consistent with $<^{\prime}$.  Set 
$\op = \frac{2}{\alpha} \omega \geq |W|$.

Since the only $k$-tuples which may become consistent with $<^{\prime}$ are those which have a nonempty intersection
with $W$, in view of (\ref{e:omega}),
   $$\oP(\exists \ \ol{K}  \not \in C(T) \ \text{consistent with} \ <^{\prime} \ | \ A_{\omega})$$
is bounded above by
\begin{align}\notag
\left(\displaystyle{\sum_{j =1}^{\op}} \dbinom{\op}{j} \dbinom{n-\op}{k-j} \right) \frac{1}{k!-1} \ &=  \ \left(\displaystyle{\sum_{j =1}^{\op} }\dbinom{\op}{j} \dbinom{n-\op}{k-j}\right) \frac{\omega}{\dbinom{n}{k}}(1+o(1))\\ \notag
& = \ \omega \sum_{j =1}^{\op} \frac{(\op)_j(n-\op)_{k-j}k!}{j!(k-j)!(n)_k}(1+o(1))\\ \label{e:inequal1}
& \leq \ \omega  \sum_{j =1}^{\op}\frac{(k)_j(\op)_j}{j!(n)_j}(1 + o(1)).
\end{align}
The inequality in (\ref{e:inequal1}) holds because $j \leq \op$ and so $\dfrac{(n - \op)_{k-j}}{(n)_k} \le \dfrac{1}{(n)_j}.$

It is straightforward to argue that for all numbers $a, b, c$ satisfying $1 \le a, b < c$, 
$$\frac{(a - 1)(b - 1)}{c - 1} < \frac{ab}{c}.$$
Repeated application of this shows that the expression in (\ref{e:inequal1}) is bounded above by
  \begin{align*}
        & \omega \sum_{j =1}^{\op} \frac{1}{j!}\left(\frac{k \op}{n}\right)^j(1+o(1)) \ \le \ \omega \left(e^{k\op/n} - 1\right)(1+o(1))\\
        & \leq \ \omega \frac{k\op}{n}(1+o(1))  = \ \frac{2 k \omega^2}{\alpha n }(1 +o(1)) < \frac{\alpha}{2}.
  \end{align*}
The last inequality follows since $\omega =  (\alpha/3e) k^{1/2}$ and $(k/e)^2/2 \le n$.

If $T$ has Property O then either $A_\omega$ does not occur, or $A_\omega$ does occur and some $\ol{K} \not \in C(\tilde{T})$ is
consistent with $<^{\prime}$.  Consequently, we have:
\begin{align*}
\oP(T \text{ has Property O}) & \ \leq \ \oP(A_\omega^c) + \oP(A_\omega) \oP(\exists \ \ol{K} \not \in C(T) \ \text{consistent with}  <^{\prime}  | \ A_{\omega})\\
& \ < \ \frac{\alpha}{2} + \frac{\alpha}{2} = \alpha.
\end{align*}

Hence the probability that $T$ fails to have Property O is at least $(1-\alpha)$. Since $T$ is a uniform selection from $\cT_{n,k}$, this is equivalent to saying at least $(1-\alpha)|\cT_{n,k}|$ members of $\cT_{n,k}$ fail to have Property O. 
\end{proof}

\section{A Construction, Small Values of $n$, and Problems} \label{S:prob}

We have an upper bound for $f(k)$, the minimum number of edges in $k$-graphs with Property~O, in Theorem \ref{T:ub}: for $k \ge k_0$, $f(k) \le (k^2 \ln k)k!$.  We have a construction of $k$-graphs with Property O, for all $k \ge 2$.  While these $k$-graphs have edge sets
that are larger than the upper bound obtained by the probabilistic proof in Section \ref{S:f(k)}, the hypergraphs are not unreasonably large.

For each $k \ge 2$ we construct an oriented $k$-graph $\cG_k = (V_k, \cE_k)$ that has Property O, where 
  \begin{equation}\label{e:induction}
     |V_k| = 3^{k - 1} \  \text{and}  \ |\cE_k| = 2^{2(k - 2)} \cdot 3^{^{\binom{k - 1}{2} + 1}}.
  \end{equation}
To begin, let $\cG_2 = (V_2, \cE_2)$ be an oriented 3-cycle.   It is clear that $\cG_2$ has Property O and its vertex and edge sets have
the sizes given in (\ref{e:induction}).

Here is the induction hypothesis: $\cG_k = (V_k, \cE_k)$ is an oriented  $k$-graph with Property O and satisfies the conditions in  (\ref{e:induction}).   Let $X, Y$ and $Z$ be three disjoint copies of $V_k$ and let $\cG_X = (X, \cE_X), \cG_Y = (Y, \cE_Y)$ and $\cG_Z = (Z, \cE_Z)$
each be isomorphic to  $\cG_k$.   Define  $\cG_{k+1} = (V_{k+1}, \cE_{k+1})$ as follows.  (See Figure 1.)

  \begin{itemize}
  
        \item Let $V_{k+1}  = X \cup Y \cup Z$. \\[-.2cm] 

        \item Let $\cE_{k+1}$ be comprised for these four types of $(k+1)-$tuples:\\

           \begin{description}
               \item[$T_1$]$\!\!= \{ (x, y_1, y_2, \ldots, y_k) : \   x \in X \ \text{and} \  (y_1, y_2, \ldots, y_k) = \overline{y} \in  \cE_Y\}$;\\
               \item[$T_2$]$\!\!= \{ (\overline{z}, x) :  \ \overline{z} \in  \cE_{Z} \ \text{and} \  x \in X\}$;\\
               \item[$T_3$]$\!\!= \{ (\overline{y}, z) : \ \overline{y} \in  \cE_{Y} \ \text{and} \  z \in Z\}$;\\ 
               \item[$T_4$]$\!\!= \{ (\overline{x}, y) : \  \overline{x} \in \cE_{X} \ \text{and} \ y \in Y\}$.     
         \end{description} 
         
    \end{itemize}
         
   \begin{figure}\label{F:Gk}
       \scalebox{1.1}{\includegraphics{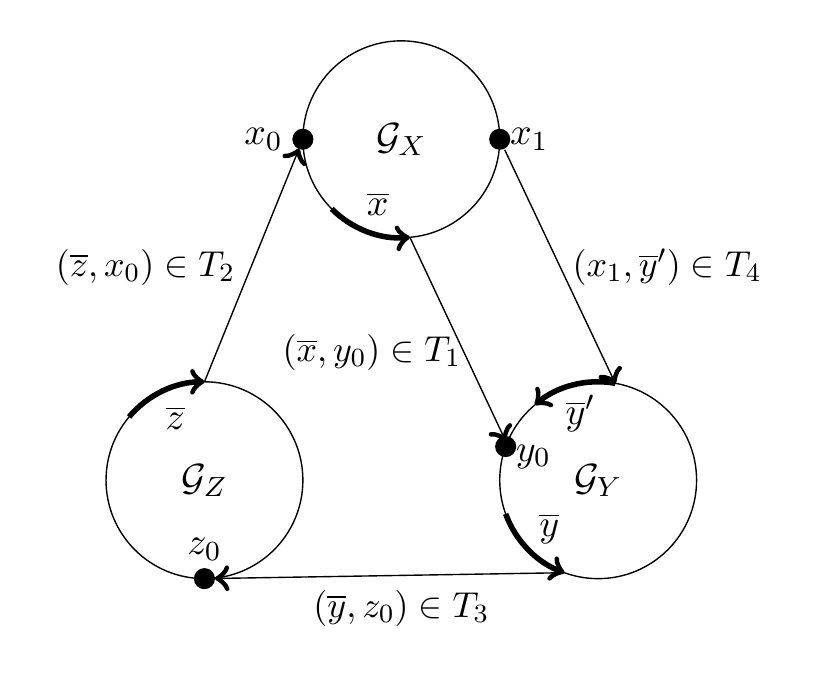}}    
       \caption{Constructing $\cG_{k+1}$ from $\cG_k$.}    
    \end{figure} 

\vspace{.5cm}
       
To see that $\cG_{k+1}$ has Property O, let $<$ be any linear order on $V_{k+1}$.  We find a member of $\cE_{k+1}$ consistent with 
$<$ as follows.

  \begin{enumerate}
  
     \item Suppose there is $x \in X$ such that $x < \min Y$.  Since  $\cG_Y$ has Property O there is some $\overline{y} \in \cE_Y$ 
     consistent with $<$.  Thus, $(x, \overline{y}) \in T_1$ is consistent with $<$.\\
     
     \item Suppose there is $x \in X$ such that $\max Z < x$.  Since  $\cG_Z$ has Property O there is some $\overline{z} \in \cE_Z$ 
     consistent with $<$.  Thus, $(\overline{z}, x) \in T_2$ is consistent with $<$.\   
     
   \end{enumerate}
   
 By {\bf(1)} and {\bf(2)}, we may assume that 
 $$\text{for all} \ x \in X \ \text{there exist} \ y_x \in Y \ \text{and} \ z_x \in Z \ \text{such that} \ y_x < x < z_x .$$
 Let $x_0 = \max X$.  Then $x \le x_0 < z_{x_0}$ for all $x \in X$.
 
  \begin{enumerate}\setcounter{enumi}{2}
  
      \item Suppose all $y \in Y$ satisfy $y < z_{x_0}$.  Since  $\cG_Y$ has Property O there is some $\overline{y} \in \cE_Y$ 
     consistent with $<$.  Then $(\overline{y}, z_{x_0}) \in T_3$ is consistent with $<$.\\
      
      \item Suppose some $y \in Y$ satisfies $z_{x_0} < y$.  Then for all $x \in X$, $x < y$.  Since $\cG_X$ has Property O there is some
      $\overline{x} \in \cE_X$ consistent with $<$.  Then $(\overline{x}, y) \in T_4$ is consistent with $<$.
   \end{enumerate}
 Therefore, $\cG_{k+1}$ has Property O. (Note that {\bf (1) - (4)} use all types of edges.)

Let us see that $\cG_{k+1} = (V_{k+1}, \cE_{k+1})$ satisfies the conditions in (\ref{e:induction}).   First, $|V_{k+1}| = 3 |V_{k}| = 3 \cdot 3^{k - 1} = 3^k$.  Second, \\[-.2cm]

\centerline{$|\cE_{k+1}| = 4 \cdot |\cE_k| \cdot |V_k| = 4 \cdot 2^{2(k - 2)} \cdot 3^{^{\binom{k - 1}{2} + 1}} \cdot 3^{k-1} = 2^{2k} \cdot 3^
{^{\binom{k}{2} + 1}}.$}

\vspace{.3cm}

Let $n(k)$ be the minimum number of vertices in a $k$-tournament with Property O.  We have already seen that for any oriented $k$-graph to have Property O, it must have at least $k!$ edges.  Since $\binom{n}{3} \ge 3!$ forces $n \ge 5$, we have $n(3) \ge 5$.  An exhaustive computer search shows that there are no $3$-tournaments on $5$ vertices with Property O.  However, the case where $n=6$ is already much more time consuming.  On the other hand, from the construction above, we have an oriented $3$-graph on $9$ vertices which has Property O.  
Thus $n(3) \le 9$.   It remains an open question as to whether there exists a $3$-tournament with Property O on $n = 6, 7$ or $8$ vertices.   So,
it is natural to pose this:\\[-.2cm]

{\bf Problem 2.}  Find the minimum number of vertices $n(3)$ in a $3$-tournament with Property~O. 

Returning to the function $f(k)$, we would like to determine $f(3)$, the minimum number of edges in an oriented 3-graph
with Property O.  It is easily seen that $f(3) > 6$.  For $k$ in general, it would be interesting to find a construction that improves the upper bound in Theorem~\ref{T:ub}.  Finally, we would like to settle Problem 1, that is, to determine whether $f(k)/k! \to \infty$.



\end{document}